\documentclass[11pt,a4paper]{article}

\usepackage{epsf,epsfig,amsfonts,amsgen,amsmath,amstext,amsbsy,amsopn,amsthm,cases,listings,color
}
\usepackage{ebezier,eepic}
\usepackage{color}
\usepackage{multirow}
\usepackage{epstopdf}
\usepackage{graphicx}
\usepackage{pgf,tikz}
\usepackage{mathrsfs}
\usepackage[marginal]{footmisc}
\usepackage{enumitem}
\usepackage[titletoc]{appendix}
\usepackage{booktabs}
\usepackage{url}
\usepackage{amssymb}
\usepackage{subfigure}
\usepackage{authblk}

\usepackage{pgfplots}

\usepackage{subfigure}
\usepackage{mathrsfs}
\usepackage{xcolor}
\usetikzlibrary{arrows}

\allowdisplaybreaks[1]

\definecolor{uuuuuu}{rgb}{0.27,0.27,0.27}
\definecolor{sqsqsq}{rgb}{0.1255,0.1255,0.1255}

\newtheorem{definition}{Definition} [section]
\newtheorem{theorem}[definition]{Theorem}
\newtheorem{lemma}[definition]{Lemma}

\newtheorem{corollary}[definition]{Corollary}
\newtheorem{conjecture}[definition]{Conjecture}

\newtheorem{fact}[definition]{Fact}
\newtheorem{question}[definition]{Question}
\newenvironment{pf}{\noindent{\bf Proof}}

\newcommand{\blow}[2]{#1(\!(#2)\!)}
\newcommand{\multiset}[1]{\{\hspace{-0.25em}\{\hspace{0.1em}#1\hspace{0.1em}\}\hspace{-0.25em}\}}

\newcommand{\hide}[1]{}
\def\multisets#1#2{\ensuremath{\left(\kern-.3em\left(\genfrac{}{}{0pt}{}{#1}{#2}\right)\kern-.3em\right)}}

\begin{document}
\title{\bf\Large Generating non-jumps from a known one}

\date{\today}
\author{Jianfeng Hou\thanks{Research was supported by National Natural Science Foundation of China (Grant No. 12071077). Email: jfhou@fzu.edu.cn}}
\author{Heng Li\thanks{Email: hengli.fzu@gmail.com}}
\author{Caihong Yang\thanks{Email: ych325123@outlook.com}}
\author{Yixiao Zhang\thanks{Email: fzuzyx@gmail.com}}
\affil{Center for Discrete Mathematics,
            Fuzhou University, Fujian, 350003, China}

\maketitle
\begin{abstract}
Let $r\ge 2$ be an integer. The real number $\alpha\in [0,1]$ is a jump for $r$ if  there exists a constant $c > 0$ such that for any $\epsilon >0$ and any integer $m \geq r$, there exists an integer $n_0(\epsilon, m)$ satisfying any $r$-uniform graph with $n\ge n_0(\epsilon, m)$ vertices and density at least $\alpha+\epsilon$ contains a subgraph with $m$ vertices and density at least $\alpha+c$. A result of Erd\H{o}s, Stone and Simonovits implies that every $\alpha\in [0,1)$ is a jump for $r=2$. Erd\H{o}s asked whether the same is true for $r\ge 3$. Frankl and R\"{o}dl gave a negative answer by showing that $1-\frac{1}{l^{r-1}}$ is not a jump for $r$ if $r\ge 3$ and $l>2r$. After that, more non-jumps are found using a method of Frankl and R\"{o}dl. In this note, we show a method to construct maps $f \colon [0,1] \to [0,1]$ that preserve non-jumps, if $\alpha$ is a non-jump for $r$ given by the method of Frankl and R\"{o}dl, then $f(\alpha)$ is also a non-jump for $r$. We use these maps to study  hypergraph Tur\'{a}n densities and  answer  a question posed by  Grosu.
\end{abstract}

\section{Introduction}\label{SEC:Introduction}
For an integer $r\ge 2$ an $r$-uniform hypergraph (henceforth $r$-graph) $\mathcal{H}$ is a collection of $r$-subsets of some finite set $V$. Let $d(\mathcal{H})=|\mathcal{H}|/\binom{|V(\mathcal{H})|}{r}$ denote the \emph{density} of $\mathcal{H}$. Given a family $\mathcal{F}$ of $r$-graphs we say $\mathcal{H}$ is $\mathcal{F}$-\emph{free}
if it does not contain any member of $\mathcal{F}$ as a subgraph.
The {\em Tur\'{a}n number} $\mathrm{ex}(n,\mathcal{F})$ of $\mathcal{F}$ is the maximum
number of edges in an $\mathcal{F}$-free $r$-graph on $n$ vertices.
The {\em Tur\'{a}n density} $\pi(\mathcal{F} )$ of $\mathcal{F}$ is defined as
$\pi(\mathcal{F}):=\lim_{n\to \infty}\mathrm{ex}(n,\mathcal{F})/{n\choose r}$; the
existence of the limit was established in~\cite{KatonaNemetzSimonovits64}.

Determining $\mathrm{ex}(n,\mathcal{F})$, which is perhaps the central topic in extremal combinatorics, is closely related to jumps. Let $r\ge 2$ be an integer. The real number $\alpha\in [0,1]$ is a $jump$ for $r$ if  there exists a constant $c > 0$ such that for any $\epsilon >0$ and any integer $m \geq r$, there exists an integer $n_0(\epsilon, m)$ satisfying any $r$-graph with $n\ge n_0(\epsilon, m)$ vertices and density at least $\alpha+\epsilon$ contains a subgraph with $m$ vertices and density at least $\alpha+c$. For graphs, the classical Erd\H{o}s-Stone-Simonovits Theorem \cite{ES1965limit, EStone1946structure} determined the Tur\'{a}n numbers of all non-bipartite graphs asymptotically, which implies that every $\alpha \in [0, 1)$ is a jump for $r = 2$. For $r\geq 3$, Erd\H{o}s \cite{erdos1964extremal} proved that every $\alpha \in [0,r!/r^r)$ is a jump for $r$ and conjectured that

\begin{conjecture}\label{erdos-conjecture}
Every $\alpha \in [0,1)$ is a jump for every $r\ge 2$.
\end{conjecture}

The conjecture was disproved by Frankl and R\"{o}dl \cite{Frankl1984hypergraphs} by showing that $1- 1/(l^r-1)$ is not a jump for $r$ if $r \geq 3$ and $l > 2r$. After that, more non-jumps  were found.  Frankl, Peng, R\"{o}dl and Talbot \cite{Frankl2007note} showed that $\frac{5r!}{2r^r}$ is a non-jump for $r \geq 3$, and for any integer $s\geq 1$ and $l\geq 9s+6$ the number $1-\frac{l}{3}+\frac{3s+2}{l^2}$ is a non-jump for $r=3$. Recently, Yan and Peng \cite{yan2021non} show that $\frac{54r!}{25r^r}$ is a non-jump for $r\geq3$ and this is the smallest known non-jump number. Meanwhile, they  gave infinitely many irrational non-jumps for every $r \geq 3$ in the same paper. For more non-jump results, please refer to \cite{peng2007non,peng2007subgraph,peng2007using,peng2008note,peng2008using,peng2009jumping,peng2008generating}. A well-known open question of Erd\H{o}s \cite{E71} is that whether $r!/r^r$ is a jump for $r\geq 3$ and what is the smallest non-jump number? Baber and Talbot \cite{baber2011hypergraphs} showed that every $\alpha \in [0.2299, 0.2316) \cup [0.2871, 8/27)$ is a jump for $r = 3$. However, both questions still remain open.

The above results about non-jumps depends on a approach given by Frankl and R\"{o}dl \cite{Frankl1984hypergraphs}: First, construct a sequence of $r$-graph $(\mathcal{G}_t)_{t=1}^{\infty}$ (call it non-jump $r$-graph sequence in Section 2) each of whose Lagrangian is greater than $\alpha/r!$, then choose a specific graph $\mathcal{G}_t$ such that the Lagrangian of its induced subhypergraphs with fixed vertices is at most $\alpha/r!$, at last use a lemma given by Frankl and R\"{o}dl \cite{Frankl1984hypergraphs} to get a contradiction, call it Frankl-R\"{o}dl method. In this note, we show a method to construct infinitely many maps $f \colon [0,1] \to [0,1]$ that preserve non-jumps and prove that

\begin{theorem}\label{THEOREM:main-theorem}
Suppose that $r\ge 3$ is an integer and $\alpha$ is a known non-jump for $r$ given by the  Frankl-R\"{o}dl method. Then we can generate infinitely many reals $f(\alpha)$ which are non-jumps for $r$.
\end{theorem}

The method used in the proof of Theorem \ref{THEOREM:main-theorem} is motivated by Liu and Pikhurko \cite{liu2022hypergraph}, who defined an operation of patterns to study hypergraph Tur\'{a}n densities. We will define the union of two patterns and show that the resulting pattern still  preserves some properties of the original one. It also has applications in  hypergraph Tur\'{a}n densities. Now, we  list one. For every integer $r\ge 2$, define
\begin{align*}
\Pi_{\mathrm{fin}}^{(r)} & := \left\{\pi(\mathcal{F}) \colon \text{$\mathcal{F}$ is a finite family of $r$-graphs} \right\}, \quad\text{and} \\
\Pi_{\infty}^{(r)} & := \left\{\pi(\mathcal{F}) \colon \text{$\mathcal{F}$ is a (possibly infinite) family of $r$-graphs} \right\}.
\end{align*}

Grosu \cite{G16} asked the following question.
\begin{question}[see Grosu \cite{G16}]\label{question-Grosu}
For any $r \geq 3$ find a polynomial $f \in \mathbb{Q}[x]$ such that for any Tur\'{a}n density $x$ for $r$-graphs, $f(x)$ is also a Tur\'{a}n density for $r$-graphs.
\end{question}

In fact, the arguments in \cite{G16} with a minor modification imply that $1-\frac{1-a}{2^{r-1}}\in \Pi_{\infty}^{(r)}$ if $a\in \Pi_{\infty}^{(r)}$. We generalize this by showing
\begin{theorem}\label{infinite-lambda}
For an integer $r \geq 2$, if $a \in \Pi_{\infty}^{(r)}$, then $1-\frac{1-a}{m^{r-1}}\in \Pi_{\infty}^{(r)}$ for each integer $m\ge 2$. \end{theorem}

 We remark that if $\alpha$ is a non-jump for $r$, then for any constant $c>0$ there exists a Tur\'{a}n density $\gamma$ such that $\gamma\in (\alpha, \alpha+c]$. By Theorem \ref{THEOREM:main-theorem} and the fact that  $f(x)=1-\frac{1-x}{m^{r-1}}$ is a linear strictly monotone increasing function on $[0,1]$, we have
\begin{corollary}\label{COROLLARY:Keep-nonjump}
If $\alpha$ is a non-jump for $r$, then for  each integer $m\geq 2$, $1-\frac{1-\alpha}{m^{r-1}}$ is also a non-jump for $r$.
\end{corollary}

The rest of the paper is organized as follow. In Section 2 we introduce some preliminary definitions and results about patterns, and define the union of patterns. In Section 3, we give a property of the union of two patterns and prove Theorem \ref{THEOREM:main-theorem}. In Section 4, we prove  Theorem \ref{infinite-lambda}.

\section{Pattern and the union}
\subsection{Pattern}
In this subsection, we introduce some preliminary definitions and results about patterns. Pikhurko \cite{PI14} defined patterns to study the possible hypergraph Tur\'{a}n densities. In this note, we give the special case of patterns. The general case can be found in \cite{PI14}.
Let  $r\ge 2$ be an integer and $S$ be a finite set. A $r$-multiset $E$ on $S$ is an unordered collection of $r$ elements from $S$ with repetitions allowed. For $i\in S$, we use $E(i)$ denote the multiplicity of $i$ in $E$.  An $r$-uniform \emph{pattern} (henceforth $r$-pattern) is a pair $P=(m,\mathcal{E})$ where $m$ is a positive integer,
$\mathcal{E}$ is a collection of $r$-multisets on $[m]$. For $i\in [m]$ and $s=0,1,\cdots,r$, let $\mathcal{E}^{i,s}=\{E\in \mathcal{E} : E(i)=s\}.$
Let $V_1,\dots,V_m$ be disjoint sets and let $V=V_1\cup\dots\cup V_m$.
The \emph{profile} of an $r$-set $S\subseteq V$ (with respect to $V_1,\dots,V_m$) is
the $r$-multiset on $[m]$ that contains element $i$ with multiplicity $|S \cap V_i|$ for every $i\in [m]$.
For an $r$-multiset $E \subseteq [m]$ let
$\blow{E}{V_1,\dots,V_m}$ consist of all $r$-subsets of $V$ whose profile is $E$.
We call this $r$-graph the \emph{blowup} of $E$ and
the $r$-graph
$$
\blow{\mathcal{E}}{V_1,\dots,V_m} := \bigcup_{E\in \mathcal{E}} \blow{E}{V_1,\dots,V_m}
$$
is called the \emph{blowup} of $\mathcal{E}$ (with respect to $V_1,\dots,V_m$).
We say an $r$-graph $\mathcal{H}$ is a \emph{$P$-construction} it is a blowup of $\mathcal{E}$.

Let $S$ be a subset of $[m]$, we use $P[S]=(|S|, \mathcal{E}[S])$ to denote the induced subpattern with respect to $S$, where $\mathcal{E}[S]=\{E\in \mathcal{E} : E\,\, {\rm is}\,\, {\rm a}\,\, {\rm multiset}\,\, {\rm on }\,\, S \}$. For $i\in [m]$ let $P-i$ be the pattern obtained
from $P$ by \emph{removing index $i$}, that is, we remove $i$ from $[m]$ and delete
all multisets containing $i$ from $\mathcal{E}$ (and relabel the remaining indices to form the set $[m-1]$).  Note that these are special cases of the more general
definitions from \cite{PI14}.

It is easy to see that the notion of a pattern is a generalization of hypergraphs,
since every $r$-graph is a pattern in which $\mathcal{E}$ is a collection of (ordinary) $r$-sets. Thus, we use $P_{\mathcal{G}}$ to denote the pattern with respect to $r$-graph $\mathcal{G}$.
For many families $\mathcal{F}$ the extremal $\mathcal{F}$-free constructions are usually a blowup of some simple patterns.
For example, let $P_{B} = (2, \left\{ \multiset{1,2,2}, \multiset{1,1,2} \right\})$ (here we use $\multiset{\empty}$ to distinguish it from ordinary sets).
Then a $P_{B}$-construction is a $3$-graph $\mathcal{H}$ whose vertex set can be partitioned into two parts $V_1$ and $V_2$
such that $\mathcal{H}$ consists of all triples that have nonempty intersections with both $V_1$ and $V_2$.
A famous result in Hypergraph Tur\'{a}n problem shows that the pattern $P_{B}$ characterizes the structure of (large) extremal constructions that do not contain a Fano plane (see \cite{DF00,FS05,KS05b}).

Denote by $\Delta_{m-1}$ the standard $(m-1)$-dimensional simplex,
i.e.,
\begin{align}
\Delta_{m-1} = \left\{(x_1,\ldots,x_m)\in [0,1]^m\colon x_1+\cdots+x_m= 1\right\}. \notag
\end{align}
For a pattern $P = (m, \mathcal{E})$ let the \emph{Lagrange polynomial} of $\mathcal{E}$ be
\begin{align*}
\lambda_{\mathcal{E}}(x_1,\dots,x_m)
:= r!\,\sum_{E\in \mathcal{E}}\; \prod_{i=1}^m\; \frac{x_i^{E(i)}}{E(i)!}.
\end{align*}
 In other words, $\lambda_{\mathcal{E}}$ gives the
asymptotic edge density of a large blowup of $\mathcal{E}$, given its relative part sizes $x_i$.
The \emph{Lagrangian} of $P$ is defined as follows:
\begin{align*}
\lambda(P) := \max\left\{\lambda_{\mathcal{E}}(x_1,\dots,x_m) \colon (x_1,\dots,x_m)\in \Delta_{m-1} \right\}.
\end{align*}
Since $\Delta_{m-1}$ is compact,
a well known theorem of Weierstra\ss\ implies that the restriction of $x$ to $\Delta_{m-1}$
attains a maximum value. Thus $\lambda(P)$ is well-defined. We call $P$ \emph{minimal} if $\lambda(P-i)$ is strictly smaller than $\lambda(P)$ for every $i\in [m]$. Note that if $P$ is minimal and $x=(x_1,\ldots,x_m)\in \Delta_{m-1}$ is a vector that maximizes $\lambda_{\mathcal{E}}$, then $x_i>0$ for $i\in [m]$. In particular, for an $r$-graph $G$, the \emph{Lagrangian} of $\mathcal{G}$ is  defined as  $\lambda(\mathcal{G})=\lambda(P_{\mathcal{G}})$.

\begin{fact}
Let $P=(m,\mathcal{E})$ be a $r$-pattern. For $S\subseteq [m]$, $\lambda(P)\ge \lambda(P[S])$.
\end{fact}

The following  lemma  is simple, whose short proof is included here for the sake of completeness.
\begin{lemma}\label{LEMMA:lambda(P)>lambda(G)}
Let $P$ be a $r$-pattern. If $\mathcal{G}$  is a $P$-construction, then $\lambda(P)\geq \lambda(\mathcal{G})$.
\end{lemma}
\begin{proof}
Let $P=(m, \mathcal{E})$ be a $r$-pattern and $\mathcal{G}$ be a $P$-construction. Notice that for a $r$-graph $\mathcal{G}$ (or $r$-pattern $P$) $\lambda(\mathcal{G})$ (or $\lambda(P)$) equals the maximum edge density of its blowup. So, it suffices to show any blowup of $\mathcal{G}$ is a subgraph of some $P$-construction on $|V(\mathcal{G})|$ vertices. Since $\mathcal{G}$ is a $P$-construction, $\mathcal{G}$ has a partition $V(\mathcal{G})=V_1\cup\dots\cup V_m$ such that for each $e\in \mathcal{G}$, the profile of $e$ with respect to $V_1, \dots, V_m$ is contained in $\mathcal{E}$. Suppose that $\mathcal{G}'$ is a blowup of $\mathcal{G}$ where each vertices $v$ in $\mathcal{G}$ is replaced by a set $W_v$ in $\mathcal{G}'$. Let $V_i'=\cup_{v\in V_i}W_v$. Then $V_1',\dots, V_m'$ is a partition of $V(\mathcal{G}')$ and the  profile of each $e\in \mathcal{G}'$ with respect to $V_1', \dots, V_m' $ is also contained in $\mathcal{E}$. This means that $\mathcal{G}'$ is a subgraph of the $P$-construction $\mathcal{E}(V'_1, \dots, V_m')$.
\end{proof}

We remark that the inequality in Lemma \ref{LEMMA:lambda(P)>lambda(G)} can be strictly true. For example, consider the 3-pattern $P=(2, \{ \multiset{1,1,2} \})$ and a $P$-construction $\mathcal{G}$ with exactly one edge $\{v_1, v_2, v_3\}$. Clearly,
\[
\lambda(P)=\max \{3x_1^2x_2 : (x_1, x_2)\in \Delta_1\}=4/9>2/9=\lambda(\mathcal{G}).
\]

Now we give some results about non-jump, Frankl and R\"{o}dl  \cite{Frankl1984hypergraphs}  gives a necessary and sufficient condition for a number $\lambda_0$ to be a jump.
\begin{lemma}\label{LEMMA:jump_condition}
The following two properties are equivalent.
\begin{itemize}
  \item [(a)] $\lambda_0$ is a jump for $r$.
 \item [(b)] There exists some finite family $\mathcal{F}$ of $r$-uniform graphs satisfying $\pi(\mathcal{F})\leq \lambda_0$ and $\lambda(F)>\lambda_0$ for all $F\in \mathcal{F}$.
\end{itemize}
\end{lemma}

\begin{definition}\label{DEFINITION:(K,lambda,epsilon)}
Let $r\ge 2$ be an integer, and $(P_t)_{t=1}^{\infty}$ be a sequence of $r$-patterns. We call $(P_t)_{t=1}^{\infty}$ is a $(k, \lambda_0)$-sequence if there exists a positive real number sequence $(\epsilon(t))_{t=1}^{\infty}$ such that the following holds:
\begin{itemize}
  \item [(1)] $\lim\limits_{t\rightarrow \infty}\lambda(P_t)=\lambda_0$,
  \item [(2)] $\lambda(P_t)\geq \lambda_0+\epsilon(t)$ for all $t$, and
  \item [(3)] $\lambda(P_t[S])\leq \lambda_0$ for every $t$ and subset $S\in \binom{[m_t]}{k}$.
\end{itemize}
\end{definition}

We call an $r$-pattern sequence $(P_t)_{t=1}^{\infty}$ is  \emph{non-jump sequence} on $\lambda_0$ if for any constant $k>0$, there exist a constant $t(k)$ such that $(P_t)_{t=t(k)}^{\infty}$ is a $(k, \lambda_0)$-sequence. For $r$-graphs, we have corresponding definitions. We remark that  $\lambda_0$ is a non-jump  through the Frankl-R\"{o}dl method  mentioned in  Introduction is to construct a non-jump $r$-graph sequence $(\mathcal{G}_t)_{t=1}^{\infty}$ on $\lambda_0$. In the following, we show Frankl-R\"{o}dl method also holds for patterns.

\begin{lemma}\label{LEMMA:NON-JUMP-CON-ON-a=a-is-non-jump}
If there exists an $r$-pattern sequence $(P_t)_{t=1}^{\infty}$ which is a non-jump sequence on $\lambda_0$, then $\lambda_0$ is a non-jump for $r$.
\end{lemma}

\begin{proof}
By contradiction, suppose that $\lambda_0$ is a jump for $r$. By Lemma \ref{LEMMA:jump_condition}, there exists a finite collection $\mathcal{F}$ of $r$-graphs satisfying $\pi(\mathcal{F})\leq \lambda_0$ and $\lambda(F)>\lambda_0$ for all $F\in \mathcal{F}$.
Let $k=\max\{|V(F)| : F\in \mathcal{F}\}$. Since $(P_t)_{t=1}^{\infty}$ is a non-jump seuqence on $\lambda_0$, there exist $t(k)$ such that $(P_t)_{t=t(k)}^{\infty}$ is a $(k, \lambda_0)$-sequence. This means there exists a positive constant $\epsilon(t(k))$ such that $\lambda(P_{t(k)})\geq \lambda_0+\epsilon(t(k))$. So, for a sufficiently large  $n$, we can find a $P_{t(k)}$-construction $\mathcal{G}=\mathcal{E}(V_1, \dots, V_m)$ with $n$ vertices and  edge density at least $\lambda_0+\epsilon(t(k))/2$. It follows from the fact $\pi(\mathcal{F})\leq \lambda_0$ that there exists a subset $S\subseteq V(\mathcal{G})$ with $|S|=k$ such that $\mathcal{G}[S]$ contains some $r$-graph $F\in \mathcal{F}$. Suppose that the profiles of $\mathcal{G}[S]$ (with respect to $V_1, \dots, V_m$) is $S_f$. Then $\mathcal{G}[S]$ is a $P_{t(k)}[S_f]$-construction. By Lemma \ref{LEMMA:lambda(P)>lambda(G)} and the fact that $\lambda(P_{t(k)}[S_f])\le \lambda_0$,
\[
\lambda(F)\le \lambda(\mathcal{G}[S])\leq \lambda(P_{t(k)}[S_f])\le \lambda_0,
\]
 a contradiction.
\end{proof}

\subsection{Union of partterns}

In this subsection, we define the union of two patterns and study its properties. To study the hypergraph Tur\'{a}n densities that have arbitrarily large algebraic degree, Liu and Pikhurko \cite{liu2022hypergraph} give an ingenious approach by
defining an operation on the pattern. Let $r\ge 3$ and $s\ge 1$ be two integers, and  $P=(m, \mathcal{E})$ be a $r$-pattern. An $(r+s)$-pattern $P+s=(m+s, \widehat{\mathcal{E}})$ is defined in the following way: for every $E\in \mathcal{E}$ we insert the $s$-set $\{m+1,\ldots,m+s\}$ into $E$, and let $\widehat{\mathcal{E}}$ denote the resulting family of $(r+s)$-multisets. The new pattern $P+s$ has many elegant properties. For examply, if $P$ is minimal, then so is $P+s$, and $\lambda(P+s)$ can be determined completely with respect to $\lambda(P)$. Motivated by it, we define the union of two patterns by the following.

\begin{definition}\label{DEFINITION:union}
Let $r>0$ be an integer. Given two $r$-patterns $P_1=(m_1, \mathcal{E}_1)$ and $P_2=(m_2, \mathcal{E}_2)$,
the union of $P_1$ and $P_2$ on $i$, denoted by $P_1\oplus _{i} P_2$, is a $r$-pattern $(m_1+m_2-1, \widehat{\mathcal{E}})$, such that $\widehat{\mathcal{E}}$ is a collection of $r$-multisets on  $\{1,\ldots,i-1,i_1,\ldots,i_{m_2},r+1,\ldots,m_1\}$ that is defined in the following way:
\begin{itemize}
  \item [(a)] if $E=\{a_1,\ldots,a_r\}$ belongs to $\mathcal{E}_2$, then $\{i_{a_1},\ldots,i_{a_r}\}\in \widehat{\mathcal{E}}$;
  \item [(b)] if $E\in \mathcal{E}_1^{i,0}$, then $E\in  \widehat{\mathcal{E}}$;
  \item [(c)] for every $s\in [r]$ and $E\in \mathcal{E}_1^{i,s}$, $\left(E\setminus i^{(a)}\right)\cup A\in \widehat{\mathcal{E}}$ for all $a$-multisets $A$ in $\{i_1,\ldots,i_{m_2}\}$, where $i^{(a)}$ denotes the multiset $\multiset{\underbrace{i,\ldots,i}_a}$.
\end{itemize}
\end{definition}
\textbf{Remarks.}
\begin{itemize}
\item By the definition of $P_1\oplus_{i} P_2$, we know if an $r$-graph $\mathcal{H}$ is a $(P_1\oplus_{i} P_2)$-construction, then there exists a partition $V(\mathcal{H})= V_1\cup \dots \cup V_{m_1}$ such that the following holds:
\begin{itemize}
  \item [(a)] the induced subgraph $\mathcal{H}[V_{i}]$ is a $P_2$-construction, and
  \item [(b)] $\mathcal{H}\setminus \mathcal{H}[V_{i}]$ consists of all $r$-sets whose profiles are in $\mathcal{E}_1$.
\end{itemize}

\item For a set $T=\{i_1,\ldots,i_t\}\subseteq [m_1]$, we can define the \emph{union} $P_1\oplus_{T} P_2$ of $P_1$ and $P_2$ on $T$ similarly. If an $r$-graph $\mathcal{H}$ is a $(P_1\oplus_{T}P_2)$-construction, then there exists a partition $V(\mathcal{H})= V_1\cup \dots \cup V_{m_1}$ such that the following holds:
\begin{itemize}
  \item [(a)] the induced subgraph $\mathcal{H}[V_{i}]$ is a $P_2$-construction for each $i\in T$, and
  \item [(b)] $\mathcal{H}\setminus \cup_{i\in T}\mathcal{H}[V_{i}]$ consists of all $r$-sets whose profiles are in $\mathcal{E}_1$.
\end{itemize}
\end{itemize}

\begin{lemma}\label{OBSERVATION:lambda(P+P_t)}
Let  $P_1\oplus _{i} P_2 = (m_1+m_2-1, \widehat{\mathcal{E}})$ be the union of two $r$-patterns $P_1=(m_1, \mathcal{E}_1)$ and $P_2=(m_2, \mathcal{E}_2)$ on $i\in[m_1]$, and $x=(x_1, \dots,x_{i-1}, x_{i_1} \dots,x_{i_{m_2}}, x_{i+1}\dots, x_{m_1})\in \Delta_{m_1+m_2-1}$. Then
\[
\lambda_{\widehat{\mathcal{E}}}(x)=\lambda_{\mathcal{E}_1}\left(x_1, \dots, x_{m_1}\right)+\lambda_{\mathcal{E}_2}(x_{i_1}, \dots, x_{i_{m_2}}),
\]
where $x_i=\sum_{j=1}^{m_2}x_{i_j}$.
\end{lemma}
\begin{proof}
For a multisets $A$ on $\{i_1,\ldots,i_{m_2}\}$, let $x_A=\prod_{a\in A}x_a$. Then for $s\in [r]$
\begin{align}\label{sum-xA}
\sum_{A \text{ is a } s\text{-multisets  on  }  \{i_1,\ldots,i_{m_2}\}}x_A=(x_{i_1}+\ldots+x_{i_{m_2}})^s=x_i^s
\end{align}

Considering the $r$-multisets in $\widehat{\mathcal{E}}$, we have
\begin{align}\label{lambda-hat-E}
\lambda_{\widehat{\mathcal{E}}}(x)=r!\left(f(x)+g(x)\right)
\end{align}
where
\begin{align}\label{f}
f(x)=\sum_{E=\{a_1,\ldots,a_r\}\in \mathcal{E}_2}\prod_{j=1}^{m_2}\frac{x_{i_j}^{E(j)}}{E(j)!}=\frac{\lambda_{\mathcal{E}_2}(x_{i_1}, \dots, x_{i_m})}{r!},
\end{align}

\begin{align}\label{g}
g(x)&=\sum_{E\in \mathcal{E}_1^{i,0}}\prod_{j=1}^{m_1}\frac{x_{j}^{E(j)}}{E(j)!}+ \sum_{s=1}^r\sum_{E\in \mathcal{E}_1^{i,s}}\sum_{A \text{ is a } s\text{-multisets  on  }  \{i_1,\ldots,i_{m_2}\}}\left(\prod_{j=1,j\neq i}^{m_1}\frac{x_{j}^{E(j)}}{E(j)!}\right)\times \frac{x_A}{s!}\notag \\
&=\sum_{E\in \mathcal{E}_1^{i,0}}\prod_{j=1}^{m_1}\frac{x_{j}^{E(j)}}{E(j)!}+ \sum_{s=1}^r\sum_{E\in \mathcal{E}_1^{i,s}}\prod_{j=1}^{m_1}\frac{x_{j}^{E(j)}}{E(j)!}\notag\\
&=\frac{\lambda_{\mathcal{E}_1}\left(x_1, \dots, x_{m_1}\right)}{r!}.
\end{align}
Note that the second equation in $g(x)$ hold by \eqref{sum-xA}. Combining \eqref{lambda-hat-E}, \eqref{f} and \eqref{g}, we complete the proof.
\end{proof}

For a subset we can also get the following similar result by  Lemma \ref{OBSERVATION:lambda(P+P_t)} and induction on the size of the subset.

\begin{corollary}\label{OBSERVATION:lambda(P+P_t)-subset}
Let $T\subseteq [m_1]$, $P_1=(m_1, \mathcal{E}_1)$ and $P_2=(m_2, \mathcal{E}_2)$ be two $r$-patterns. If $P_1\oplus _{T} P_2 = (m_1+|T|(m_2-1), \widehat{\mathcal{E}})$ is the union of $P_1$ and $P_2$ on $T$ and $x \in \Delta_{m_1+|T|(m_2-1)}$, then
\[
\lambda_{\widehat{\mathcal{E}}}(x)=\lambda_{\mathcal{E}_1}\left(x_1, \dots, x_{m_1}\right)+\sum_{i\in T}\lambda_{\mathcal{E}_2}(x_{i_1}, \dots, x_{i_{m_2}}),
\]
where $x_i=\sum_{j=1}^{m_2}x_{i_j}$ for each $i\in T$.
\end{corollary}

The following interesting result shows that the Lagrangian of $P_1\oplus _{i} P_2$ is not related to the structure of $P_2$.
\begin{lemma}\label{CLAIM:lambda(p_t)=varphi(lambda(p_t))}
Let  $P_1\oplus _{i} P_2 = (m_1+m_2-1, \widehat{\mathcal{E}})$ be the union of two $r$-patterns $P_1=(m_1, \mathcal{E}_1)$ and $P_2=(m_2, \mathcal{E}_2)$ on $i\in[m_1]$ with $\lambda=\lambda(P_2)$. Then
\[
\lambda(P_1\oplus_i P_2)=\max\{ \lambda_{\mathcal{E}_1}(x)+\lambda x_i^r : x=(x_1, \dots, x_{m_1})\in \Delta_{m_1-1}\}.
\]
\end{lemma}
\begin{proof}
Let
\[
f(\lambda)=\max\{ \lambda_{\mathcal{E}_1}(x)+\lambda x_i^r : x=(x_1, \dots, x_{m_1})\in \Delta_{m_1-1}\}.
\]
By Weierstra\ss's Theorem, $f(\lambda)$ is well-defined.
Firstly, we prove $\lambda(P_1\oplus_i P_2)\leq f(\lambda)$. Suppose that $x=(x_1, \dots,x_{i-1}, x_{i_1} \dots,x_{i_{m_2}}, x_{i+1}\dots, x_{m_1})\in \Delta_{m_1+m_2-1}$ satisfies $\lambda(P_1\oplus_i P_2)=\lambda_{\widehat{\mathcal{E}}}(x)$. Let $x_i=\sum_{j=1}^{m_2}x_{i_j}$. Then $(x_1, \dots, x_{m_1})\in\Delta_{m-1}$, and so by Lemma  \ref{OBSERVATION:lambda(P+P_t)},
\begin{align}\label{2.2-again}
\lambda_{\widehat{\mathcal{E}}}(x)=\lambda_{\mathcal{E}_1}\left(x_1, \dots, x_{m_1}\right)+\lambda_{\mathcal{E}_2}(x_{i_1}, \dots, x_{i_{m_2}}),
\end{align}
If $x_i=0$, then $\lambda_{\mathcal{E}_2}(x_{i_1}, \dots, x_{i_{m_2}})=0$, and so
\[
\lambda_{\widehat{\mathcal{E}}}(x)=\lambda_{\mathcal{E}_1}\left(x_1, \dots, x_{m_1}\right)\leq f(\lambda),
\]
we are done. Suppose that $x_i>0$. Then $y=(\frac{x_{i_1}}{x_i}, \dots, \frac{x_{i_{m_2}}}{x_i})\in \Delta_{m_2-1}$. Since $\lambda_{\mathcal{E}_2}(x_{i_1}, \dots, x_{i_{m_2}})$ is homogenous, we have
\[
\lambda_{\mathcal{E}_2}(x_{i_1}, \dots, x_{i_{m_2}})=
\lambda_{\mathcal{E}_2}\left(\frac{x_{i_1}}{x_i}, \dots, \frac{x_{i_{m_2}}}{x_i}\right)x_{i}^r\leq \lambda x_{i}^r,
\]
which together with \eqref{2.2-again} yields that
\[
\lambda_{\widehat{\mathcal{E}}}(x)\le \lambda_{\mathcal{E}_1}\left(x_1, \dots, x_{m_1}\right)+\lambda x_{i}^r\leq f(\lambda).
\]

Next, we prove the other side, i.e., $\lambda(P_1\oplus_i P_2)\geq f(\lambda)$. Suppose that $y=(y_1,\dots, y_{m_1})$ is a vector in $\Delta_{m_1-1}$ such that $\lambda_{\mathcal{E}}(y)+\lambda y_i^r=f(\lambda)$. If $y_i=0$, then
\[
f(\lambda)=\lambda_{\mathcal{E}_1}(y)\leq \lambda(P_1)\leq \lambda(P_1\oplus_i P_2),
\]
we are done. Assume that $y_i>0$. Let $z=(z_1, \dots, z_{m_2})$ be a vector in $\Delta_{m_2-1}$ that maximizes $\lambda_{\mathcal{E}_2}(x)$. For $j\in [m_1]$ and $j\neq i$, let $x_i=y_i$. For $j\in[m_2]$, let $x_{i_j}=y_i\times z_{j}$.
Then $x=(x_1, \dots,x_{i-1}, x_{i_1} \dots,x_{i_{m_2}}, x_{i+1}\dots, x_{m_1})\in \Delta_{m_1+m_2-1}$. Again, by Lemma \ref{OBSERVATION:lambda(P+P_t)}, we have
\begin{align}
\lambda(P_1\oplus_i P_2)&\geq \lambda_{\widehat{\mathcal{E}}}(x)\notag\\
&=\lambda_{\mathcal{E}_1}\left(x_1, \dots, x_{m_1}\right)+\lambda_{\mathcal{E}_2}(x_{i_1}, \dots, x_{i_{m_2}})\notag\\
&=\lambda_{\mathcal{E}_1}\left(x_1, \dots, x_{m_1}\right)+\lambda_{\mathcal{E}_2}\left(z_1,\dots, z_{m_2}\right)y_i^r\notag\\
&=\lambda_{\mathcal{E}}\left(y_1,\dots,y_{m}\right)+\lambda y_i^r \notag\\
&=f(\lambda).\notag
\end{align}
This completes the proof.
\end{proof}

If we use Corollary \ref{OBSERVATION:lambda(P+P_t)-subset} and the idea in the proof of Lemma \ref{CLAIM:lambda(p_t)=varphi(lambda(p_t))}, then we have

\begin{corollary}\label{CLAIM:lambda(p_t)=varphi(lambda(p_t))-subset}
Let  $P_1\oplus _{T} P_2 = (m_1+m_2-1, \widehat{\mathcal{E}})$ be the union of two $r$-patterns $P_1=(m_1, \mathcal{E}_1)$ and $P_2=(m_2, \mathcal{E}_2)$ on $T\subseteq [m_1]$ with $\lambda=\lambda(P_2)$. Then
\[
\lambda(P_1\oplus_T P_2)=\max\left\{ \lambda_{\mathcal{E}_1}(x)+\lambda \left(\sum_{i\in T}x_i^r\right) : x=(x_1, \dots, x_{m_1})\in \Delta_{m_1-1}\right\}.
\]
\end{corollary}

\section{The proof of Theorem \ref{THEOREM:main-theorem}}

In this section, we prove Theorem \ref{THEOREM:main-theorem} using the following result.
\begin{theorem}\label{THEOREM:new_nonjump}
Let  $P=(m, \mathcal{E})$ be a minimal pattern with $\lambda(P)< 1$, and  $(P_t=(m_t, \mathcal{E}_t))_{t=1}^{\infty}$ be a $(k, \lambda_0)$-sequence. Then $(P\oplus_i P_t)_{t=1}^{\infty}$ is a $(k, f(\lambda_0))$-sequence for every $i\in [m]$, where
\[
f(\lambda_0)=\max\{ \lambda_{\mathcal{E}}(x)+\lambda_0x_m^r : x=(x_1, \dots, x_m)\in \Delta_{m-1}\}.
\]
\end{theorem}

\begin{proof}
By symmetry, it suffices to show $(P\oplus_m P_t)_{t=1}^{\infty}$ is a $(k, f(\lambda_0))$-sequence. Let $P\oplus_m P_t=(m+m_t-1, \widehat{\mathcal{E}}_t)$. For $x=(x_1, \dots, x_m)\in \Delta_{m-1}$ and $\lambda\in [0, 1)$, we define the function
\begin{align}\label{EQUATION:phi(x,lambda)}
\phi(x, \lambda)=\lambda_{\mathcal{E}}(x)+\lambda x_m^r.
\end{align}
For $\lambda_1, \lambda_2\in [0,1)$,
\[
f(\lambda_1)=\max_{x\in \Delta_{m-1}}\Big(\phi(x, \lambda_1)-\phi(x, \lambda_2)+\phi(x, \lambda_2)\Big)\le \max_{x\in \Delta_{m-1}}|\phi(x, \lambda_1)-\phi(x, \lambda_2)|+f(\lambda_2).
\]
By symmetry,
\[
f(\lambda_2)\le \max_{x\in \Delta_{m-1}}|\phi(x, \lambda_1)-\phi(x, \lambda_2)|+f(\lambda_1).
\]
This means that
\[
|f(\lambda_1)-f(\lambda_2)|\le \max_{x\in \Delta_{m-1}}|\phi(x, \lambda_1)-\phi(x, \lambda_2)|\le |\lambda_1-\lambda_2|.
\]
Thus, $f(\lambda)$ is a  continuous function of $\lambda\in [0, 1)$, which together with $\lim_{t\rightarrow \infty}\lambda(P_t)=\lambda_0$ yields
\begin{align}\label{continuousfunction-f}
\lim_{t\rightarrow \infty}f(\lambda(P_t))=f\left(\lim_{t\rightarrow \infty}\lambda(P_t)\right)=f(\lambda_0).
\end{align}
Combining \eqref{continuousfunction-f} and Lemma  \ref{CLAIM:lambda(p_t)=varphi(lambda(p_t))}, we have $\lim_{t\rightarrow \infty}\lambda(P\oplus_m P_t)=f(\lambda_0)$.

Now we verify (2) of Definition \ref{DEFINITION:union}.
Recall that  $(P_t)_{t=1}^{\infty}$ is a $(k, \lambda_0)$-sequence. There exists a positive real number sequence $(\epsilon(t))_{t=1}^{\infty}$ such that for all $t$, $\lambda(P_t)\geq \lambda_0+\epsilon(t)$. Notice that $\phi(x, \lambda)$ is  linear as a function of  $\lambda$. The function $f(\lambda)$ is non-decreasing. By Claim \ref{CLAIM:lambda(p_t)=varphi(lambda(p_t))}, we have
\begin{align}\label{EQUATION:lambda_(P+P_t)>lambda_(lambda_0+c)}
\lambda(P\oplus_m P_t)=f(\lambda(P_t))\geq f(\lambda_0+\epsilon(t)).
\end{align}
So if  $x=(x_{1}\dots, x_{m})\in \Delta_{m-1}$ is a vector that maximizes $\phi(x, \lambda_0)$, then
\begin{align}\label{EQUATION:lambda(P+P_t)>varphi(lambda_0)+C'}
\lambda(P\oplus_m P_t)\geq f(\lambda_0+\epsilon(t))\geq \phi(x, \lambda_0+\epsilon(t))= f(\lambda_0)+\epsilon(t)x_{m}^r.
\end{align}
Notice that $x_{m}>0$, since otherwise,  $\phi(x, \lambda_0)=\lambda_{\mathcal{E}}(x)< \lambda(P)$ by the fact that $P$ is  minimal, a contradiction. Thus, $\lambda(P\oplus_m P_t)\geq f(\lambda_0)+\epsilon'(t)$ by letting $\epsilon'(t)=\epsilon(t)x_{m}^r$.

In the following, we prove that $\lambda((P\oplus_m P_t)[S])\leq f(\lambda_0)$ for every $S\in \binom{[m+m_t-1]}{k}$. Fix a subset $S=\{s_1, \dots, s_k\}\in \binom{[m+n_t-1]}{k}$. Let $(P\oplus_m P_t)[S]=(k, \widehat{\mathcal{E}}_t[S])$. Now we show $\lambda((P\oplus_m P_t)[S])\leq f(\lambda_0)$. Let $(x_{s_1}, \dots, x_{s_k})\in \Delta_{k-1}$ be a vector that maximizes $\lambda_{\widehat{\mathcal{E}}_t[S]}(x_{s_1}, \dots, x_{s_k})$. Let $x'_i=x_i$ for $i\in S$, $x'_i=0$ for $i\in [m+n_t-1]\setminus S$, and $y=\sum_{i=m}^{m+m_t-1}x'_i$. Then $(x'_1, \dots, x'_{m+m_t-1})\in \Delta_{m+m_t-2}$, and by Observation \ref{OBSERVATION:lambda(P+P_t)}, we have
\begin{align}\label{EQUATION:lambda_(P_S)}
\lambda((P\oplus_m P_t)[S]) &=\lambda_{\widehat{\mathcal{E}}_{t}[S]}(x_{s_1},\dots, x_{s_k})\notag\\
&=\lambda_{\widehat{\mathcal{E}}_t}(x'_1,\dots, x'_{m+m_t-1})\notag\\
&=\lambda_{\mathcal{E}}(x'_1, \dots, x'_{m-1},y)+\lambda_{\mathcal{E}_{t}}\left(x'_{m},\dots, x'_{m+m_t-1}\right).
\end{align}
Again, if $y=0$, then $\lambda((P\oplus_m P_t)[S])\le \lambda_{\mathcal{E}}(x'_1, \dots, x'_{m-1},y)\le f(\lambda_0)$, we are done. Otherwise, it follows from  $|S\cap [m, m+m_t-1]|\leq k$ and $\lambda(P_t[S])\leq \lambda_0$ that
\[
\lambda_{\mathcal{E}_{t}}\left(\frac{x'_{m}}{y},\dots, \frac{x'_{m+m_t-1}}{y}\right)\leq \lambda(P_t[S])\leq \lambda_0,
 \]
which together with \eqref{EQUATION:lambda_(P_S)} yields
\begin{align}
\lambda((P\oplus_m P_t)[S])\leq \lambda_{\mathcal{E}}(x'_1, \dots, x'_{m-1}, y)+\lambda_0y^r\leq f(\lambda_0).\notag
\end{align}

All the arguments above shows that $(P\oplus_m P_t)_{t=1}^{\infty}$ is a $(k, f(\lambda_0))$-sequence, which completes the proof of Theorem \ref{THEOREM:new_nonjump}.
\end{proof}

Now we prove Theorem  \ref{THEOREM:main-theorem}.
\begin{proof}[Proof of Theorem  \ref{THEOREM:main-theorem}]
For a minimal pattern $P=(m, \mathcal{E})$, let
\[
f_i(\lambda)=\max\{\lambda_{\mathcal{E}}(x)+\lambda x_i^r \colon x=(x_1, \dots, x_m)\in \Delta_{m-1}\}
\]
Suppose that $\alpha$ is a non-jump for $r$ given by the Frankl-R\"{o}dl method. Then there exists a  non-jump sequence $(\mathcal{G}_t)_{t=1}^{\infty}$ on $\lambda_0$, then we have $(P\oplus_i P_{\mathcal{G}_t})_{t=1}^{\infty}$ is a non-jump sequence on $f_i(\alpha)$ for every $i\in [m]$ by Theorem \ref{THEOREM:new_nonjump}.
By Lemma \ref{LEMMA:NON-JUMP-CON-ON-a=a-is-non-jump}, $f_i(\alpha)$ is a non-jump for each $i\in[m]$, which completes the proof of Theorem \ref{THEOREM:main-theorem}.
\end{proof}
We end this section by the following remarks.
\begin{itemize}
  \item Let $m,r$ be integers with $m\ge r\ge 2$, and let $P^m=(m,\mathcal{E}^m)$  be a $r$-pattern by letting $\mathcal{E}^m$ be a collection of all $r$-sets of $[m]$. This means that if $\mathcal{G}$ is a  $P^m$-construction if and only if $\mathcal{G}$ is a complete $m$-partite $r$-graph. For $x=(x_1, \dots, x_m)\in \Delta_{m-1}$ and $S\in \binom{[m]}{r}$, let $x_S=\prod_{i\in S}x_i$. Then $\lambda_{\mathcal{E}^m}(x)=\sum_{S\in \binom{[m]}{r}}x_S$. For every non-jump number $\alpha$ obtained by the Frankl-R\"{o}dl method, we have
\[
f_m(\alpha)=\max_{x\in \Delta_{m-1}}\left\{ \sum_{S\in \binom{[m]}{r}}x_S+\alpha x_m^r\right\}
\]
is a non-jump for $r$.
  \item Suppose $\alpha$ is a known non-jump number obtained by the Frankl-R\"{o}dl method. If a minimal pattern $P$ in the proof of Theorem \ref{THEOREM:main-theorem} satisfies $\lambda(P)\leq \alpha$, then the resulting value $f(\alpha)$ maybe equal  $\alpha$. However, if $\lambda(P)>\alpha$, then we can obtain a non-jump number $f(\alpha)$ that is strictly greater than $\alpha$.
  \item We can replace the $i\in [m]$ in the proof of Theorem \ref{THEOREM:main-theorem} with the subset $T\subseteq[m]$, and then we can get more functions.
\end{itemize}

\section{An application to Tur\'{a}n density}

In Section 2, we have defined the union of two patterns, and studied its properties. Beside finding non-jumps, it has many applications to Tur\'{a}n densities. In this section, we list one and prove Theorem \ref{infinite-lambda}. First, we give some fundamental theorems on Tur\'{a}n densities.  Define
\begin{align*}
\Lambda ^{(r)} := \left\{\lambda(\mathcal{G}) \colon \text{$\mathcal{G}$ is an $r$-graph} \right\}
\end{align*}
for every integer $r\ge 2$.
Pikhurko \cite{PI14} proved the following result.

\begin{theorem}[see Pikhurko \cite{PI14}]
$\Lambda ^{(r)} \subseteq \Pi_{\mathrm{fin}} ^{(r)}$.
\end{theorem}

It was shown by Brown and Simonovits \cite{BS85} that $\Lambda ^{(r)}$ is dense in $\Pi_\infty ^{(r)}$. As the latter is a closed set, Grosu \cite{G16} showed that the closure of $\Lambda ^{(r)}$ is $\Pi_\infty ^{(r)}$. Namely, it was proved that

\begin{lemma}[see Grosu \cite{G16}]\label{Lemma:closed}
$\overline{\Lambda}^{(r)} =  \Pi_\infty ^{(r)}$.
\end{lemma}
Next we apply Lemma \ref{Lemma:closed} to prove Theorem \ref{infinite-lambda}

\begin{proof}[Proof of Theorem \ref{infinite-lambda}]
Let $r\ge 2$ be an integer. By Lemma \ref{Lemma:closed}, it suffices to show that if $a \in \overline{\Lambda}^{(r)}$, then $1-\frac{1-a}{m^{r-1}}\in \overline{\Lambda}^{(r)}$ for each integer $m\ge 2$.

For $m\ge 2$, let $P = (m, \mathcal{E})$ be a $r$-pattern with $\mathcal{E}=[m]^{(r)}\setminus \{\multiset{i^{(r)}}:i\in[m]\}$, where $[m]^{(r)}$ denotes all $r$-multiset on $[m]$. For $x=(x_1,\ldots,x_m)\in \Delta_{m-1}$,
\begin{align}\label{Lagrange-P-Grosu}
\lambda_{\mathcal{E}}(x)
= r!\,\sum_{E\in \mathcal{E}}\; \prod_{i=1}^m\; \frac{x_i^{E(i)}}{E(i)!}=1-\sum_{i=1}^mx_i^r.
\end{align}
For a $r$-graph $\mathcal{G}$, consider the pattern $P\oplus_{[m]} P_{\mathcal{G}}= (m|V(\mathcal{G})|, \widehat{\mathcal{E}})$. By Corollary \ref{CLAIM:lambda(p_t)=varphi(lambda(p_t))-subset} and \eqref{Lagrange-P-Grosu}, we have
\begin{align}\label{lambda-union-infinite}
\lambda(P\oplus_{[m]} P_{\mathcal{G}})
&=\max\left\{ \lambda_{\mathcal{E}}(x)+\lambda(\mathcal{G}) \left(\sum_{i=1}^mx_i^r\right) : x=(x_1, \dots, x_{m})\in \Delta_{m-1}\right\}\notag\\
&=1-(1-\lambda(\mathcal{G}))\max_{x\in \Delta_{m-1}}\sum_{i=1}^mx_i^r\notag\\
&=1-\frac{1-\lambda(\mathcal{G})}{m^{r-1}}.
\end{align}
Thus, if $a\in \Lambda^{(r)}$, then $1-\frac{1-a}{m^{r-1}}\in \Lambda^{(r)}$. Suppose that $a\in \overline{\Lambda}^{(r)}\setminus \Lambda^{(r)}$. Then there exists a sequence of $r$-graphs $(\mathcal{G}_t)_{t=1}^{\infty}$ such that $\lim_{t\rightarrow \infty}\lambda(\mathcal{G}_t)=a$. Considering  the pattern $P\oplus_{[m]} P_{\mathcal{G}_t}$, we have $1-\frac{1-\lambda(\mathcal{G}_t)}{m^{r-1}}\in \Lambda^{(r)}$, which together with
\[
\lim_{t\rightarrow \infty}\left(1-\frac{1-\lambda(\mathcal{G}_t)}{m^{r-1}}\right)=1-\frac{1-a}{m^{r-1}}
\]
yields that $1-\frac{1-a}{m^{r-1}}\in \overline{\Lambda}^{(r)}$.
\end{proof}

~~

\textbf{Acknowledgements.} The authors would like to thank Dr. Xizhi Liu for bringing this topic to our attention and his fruitful discussions.

\bibliographystyle{abbrv}
\bibliography{nonjump}
\end{document}